\documentclass[hidelinks,onefignum,onetabnum]{siamart220329}

\usepackage{lipsum}
\usepackage{amsfonts}
\usepackage{graphicx}
\usepackage{epstopdf}

\usepackage{tikz, multirow, makecell, booktabs}
\usetikzlibrary{shapes.geometric, arrows}

\usepackage{bm,enumitem,comment}
\usepackage[linesnumbered,ruled,vlined]{algorithm2e}

\SetAlgoCaptionSeparator{ }

\usepackage[left=1.72in, right=1.65in, top=1.37in, bottom=1.37in]{geometry}

\newcommand{\vertiii}[1]{{\left\vert\kern-0.25ex\left\vert\kern-0.25ex\left\vert #1 \right\vert\kern-0.25ex\right\vert\kern-0.25ex\right\vert}}

\ifpdf
  \DeclareGraphicsExtensions{.eps,.pdf,.png,.jpg}
\else
  \DeclareGraphicsExtensions{.eps}
\fi

\newsiamremark{remark}{Remark}
\newsiamremark{hypothesis}{Hypothesis}
\crefname{hypothesis}{Hypothesis}{Hypotheses}
\newsiamthm{claim}{Claim}

\headers{Deep collocation method with error control}{M. Weng, Z. Mao, and J. Shen}

\title{Deep collocation method: A framework for solving PDEs using neural networks with error control 
}

\author{Mingxing Weng\thanks{School of Mathematical Sciences, Shanghai Jiao Tong University,
	Shanghai 200240, China; School of Mathematical Science, Eastern Institute of Technology, Ningbo, 
	Zhejiang 315200, China
	(mxweng22@sjtu.edu.cn)}
\and Zhiping Mao\thanks{School of Mathematical Science, Eastern Institute of Technology, Ningbo, 
		Zhejiang 315200, China
  		(zmao@eitech.edu.cn, jshen@eitech.edu.cn).}
\and Jie Shen\footnotemark[2]
}

\usepackage{amsopn}

\graphicspath{{fig/}}

\ifpdf
\hypersetup{
  pdftitle={Deep collocation methods},
  pdfauthor={Z. Mao, J. Shen, and M. Weng}
}
\fi

\begin{document}

\maketitle

\begin{abstract}
	Neural networks have shown significant potential in solving partial differential equations (PDEs). While deep networks are capable of approximating complex functions, direct one-shot training often faces limitations in both accuracy and computational efficiency. To address these challenges, we propose an adaptive method that uses single-hidden-layer neural networks to construct basis functions guided by the equation residual. The approximate solution is computed within the space spanned by these basis functions, employing a collocation least squares scheme. As the approximation space gradually expands, the solution is iteratively refined; meanwhile, the progressive improvements serve as reliable {\it a posteriori} error indicators that guide the termination of the sequential updates. Additionally, we introduce adaptive strategies for collocation point selection and parameter initialization to enhance robustness and improve the expressiveness of the neural networks. We also derive the approximation error estimate and validate the proposed method with several numerical experiments on various challenging PDEs, demonstrating both high accuracy and robustness of the proposed method.
\end{abstract}



\section{Introduction}

Partial differential equations (PDEs) are ubiquitous in modeling various physical, biological, and engineering processes. 
Traditional numerical methods, such as finite difference methods (FDM), finite element methods (FEM), and spectral methods, have long been the standard tools for solving PDEs. These techniques have achieved substantial success in scientific computing and have been applied extensively to a wide range of scientific and engineering problems.

However, traditional numerical methods suffer from a significant limitation known as the curse of dimensionality. These methods typically discretize the problem domain into grids or meshes. While effective for low-dimensional problems, the number of grid points required for accurate solutions grows exponentially with  dimensions, leading to prohibitive computational storage and cost. Additionally, generating appropriate meshes can be time-consuming, particularly for problems with complex geometries. Consequently, mesh-based methods become impractical for high-dimensional  problems. Moreover, traditional mesh-based methods lack flexibility, requiring 
significant structural modifications to accommodate changes in spatial dimensions or PDE types.


To overcome these challenges, machine learning-based methods have emerged as a promising alternative. Neural networks, in particular, provide a mesh-free approach to solving PDEs, effectively bypassing the curse of dimensionality~\cite{han2018solving} and offering greater flexibility in handling complex geometries. Techniques like physics-informed neural networks (PINN) \cite{Raissi2019PhysicsinformedNN, Lu2021PhysicsinformedNN}, Deep Ritz Method (DRM) \cite{Weinan2017TheDR}, Weak Adversarial Networks (WAN) \cite{Zang2019WeakAN}, Deep Galerkin
Method (DGM) \cite{sirignano_dgm_2018} have gained popularity for their ability to approximate PDE solutions without the need for predefined meshes. \emph{However, despite their potential applications in engineering, machine learning methods still face limitations in terms of accuracy and computational efficiency.} Training neural networks involves solving a highly nonlinear and nonconvex optimization problem, which can be particularly challenging when  high accuracy is needed. In addition, the convergence is usually not guaranteed.

In recent years, researchers have introduced various strategies to enhance the accuracy and efficiency of neural network-based solvers for PDEs. A significant focus has been on improving PINN, with a wide range of advancements, including adaptive sampling techniques for handling sharp solutions~\cite{mao2020physics, Gao2022FailureinformedAS, mao2023physics} or high-dimensional problems~\cite{tang2023pinns, zhang2025annealed}, homogenization methods for multiscale PDEs~\cite{leung2022nh}, and second-order optimization methods to support multitask learning~\cite{wang2025gradient}, among others. 
Beyond the PINN framework, two particularly noteworthy high-accuracy methods include the randomized neural networks based approaches \cite{huang2006extreme,dong2021local,chen_bridging_2022,zhang_transferable_2024,dang_adaptive_2024}, which aim to reduce optimization errors and simplify training, and the residual-driven iterative methods  \cite{ainsworth_galerkin_2021, wang_multi-stage_2024, aldirany_multi-level_2023, xu_multi-grade_2023}, which progressively refine solutions through successive approximations.

Randomized Neural Networks (RNN) \cite{PAO1994163,471375} offer a promising approach to reduce the optimization error introduced by training algorithms, albeit at the cost of some approximation capability. Specifically, RNN simplify the training process by fixing certain parameters and focusing optimization on the linear coefficients before the output layer, reducing the problem to a least squares solution of a linear system. A prominent example is the Extreme Learning Machine (ELM) \cite{huang2006extreme}, a single-hidden-layer RNN that has been successfully applied to solve PDEs in strong-form  \cite{yang2018novel,fabiani2021numerical,WANG2024116578}. Building on ELM, Dong et al. introduced the Local Extreme Learning Machine (locELM) \cite{dong2021local}, which integrates ELM with domain decomposition, using a local feedforward network for each subdomain. This method strikes an effective balance between computational efficiency and high accuracy solutions.
The Random Feature Method (RFM), introduced by Chen et al. \cite{chen_bridging_2022}, extends locELM by replacing domain decomposition with the Partition of Unity (PoU) technique. RFM incorporates multiscale basis functions into the network and includes adaptive weight scaling in the loss function, achieving near-spectral convergence in experiments, including a challenging problem in porous geometry.

The aforementioned RNN-based methods deliver high accuracy for smooth solutions, but they struggle with problems having singularities or sharp interfaces. In addition, both locELM and RFM \emph{heavily rely on the initialization of hidden layer parameters}. To mitigate the need for manual tuning, Zhang et al. \cite{zhang_transferable_2024} proposed a transferable hidden layer parameter pre-training method that enhances network expressiveness by fitting randomly generated Gaussian random fields. Based on a similar initialization strategy, Wang et al. \cite{dang_adaptive_2024} developed the Adaptive Growing Randomized Neural Network (AG-RNN), which dynamically adjusts the network size throughout the iterative solution process. By constructively adding neurons based on the residuals from previous steps, AG-RNN enhances its capacity to represent complex solutions without introducing excessive additional optimization steps. AG-RNN has demonstrated high accuracy in solving PDEs with sharp or discontinuous solutions, offering an efficient and adaptive framework for tackling challenging problems.

These RNN-based methods offer a sophisticated blend of neural networks' representational power and traditional numerical techniques' robustness. By constructing approximation spaces with neural networks and applying least squares methods during training, they strike a balance between computational efficiency and solution accuracy.
Nevertheless, all aforementioned RNN-based methods have the fundamental issue of \emph{illconditioning}.

Beyond RNN-based approaches, residual-based methods have also garnered considerable interest. By minimizing the residuals at each step, these methods achieve progressive improvement in the solution, often leading to highly accurate results. 
The Galerkin Neural Networks (GNN), developed by Ainsworth et al. \cite{ainsworth_galerkin_2021, ainsworth_galerkin_2022}, proposed a residual-driven iterative scheme in the weak form of PDEs. This method provides an error estimator as a stopping criterion, ensuring high accuracy. The Extended Galerkin Neural Networks (xGNN) \cite{ainsworth_extended_2024} further generalize this approach to handle broader boundary value problems, particularly improving accuracy in the presence of singularities. However, a key limitation of GNN is its reliance on the mesh-based Gaussian quadrature for numerical integration, which significantly increases computational cost and leads to the curse of dimensionality in high-dimensional problems. In addition, GNN faces challenges in handling nonlinear problems and problems in complex domains.
In contrast to GNN, which works with the weak form of PDEs, Multi-stage Neural Networks (MSNN) by Wang et al. \cite{wang_multi-stage_2024} and the Multi-level Neural Networks (MLNN) by Aldirany et al. \cite{aldirany_multi-level_2023} apply residual corrections in the strong form of PDEs. Both methods introduce new networks at each iteration, specifically tailored to reduce the residual of the previous stage. With appropriate initialization, these approaches have achieved substantial error reduction, reaching machine precision in some cases.
Another development is the Multigrade Neural Network (MGNN) proposed by Xu et al. \cite{xu_multi-grade_2023, xu_multi-grade_2023-1}. MGNN deepens the network structure with each iteration while preserving fixed parameters from earlier stages, simplifying optimization and enhancing accuracy. Nevertheless, its training process exhibits inconsistent error reduction, making convergence rates less predictable. Furthermore, MSNN, MLNN, and MGNN all suffer from prolonged training times due to their reliance on the Adam optimizer, which requires a large number of optimization steps.


The aim of this work is to propose a collocation framework for solving a wide range of PDEs using neural networks with error control. 
In particular, we follow the idea of GNN by greedily constructing a sequence of finite-dimensional subspaces, where the basis functions are generated through a series of neural networks. Our approach incorporates adaptive strategies for network initialization and operates within a collocation framework, leveraging a least squares formulation for network training. This design significantly enhances \emph{computational efficiency, accuracy, and flexibility}.

The proposed method enjoys several key advantages:
\begin{itemize}[left=0pt] 
\item \emph{Robustness, efficiency and high accuracy:} We integrate several adaptive strategies to reduce the difficulty in training neural networks and alleviate the challenges of hyperparameter tuning. Moreover, the method maintains stable performance when hyperparameters are chosen within a suitable range, ensuring \emph{robustness, efficiency and accuracy}.

\item \emph{Guaranteed convergence:} We establish a convergence property demonstrating that the error remains non-increasing throughout iterations. Under appropriate network optimization assumptions, we derive a reliable {\it a posteriori} error estimator and provide an error estimate, offering theoretical insights into achieving \emph{geometric convergence}. These findings are further validated by numerical experiments.

\item \emph{Broader applicability:} 
Compared with GNN, the proposed collocation-based method exhibits broader applicability across various problems, such as \emph{nonlinear problems and problems in complex domains}. 

\end{itemize}

The remainder of this paper is organized as follows:
In \Cref{sec:methodology}, we introduce the proposed methodology, outlining the collocation framework, the construction of neural network basis functions, and the adaptive strategies used for the selection of collocation points and network initialization.
\Cref{sec:analysis} is dedicated to proving the convergence of the method, as well as providing an {\it a posteriori} indicator to guide the termination of the iterative process.
In \Cref{sec:numerical-results}, we present numerical experiments that assess the performance of the proposed method across a range of PDEs.
Finally, in \Cref{sec:conclusions}, we summarize our findings and suggest possible avenues for future research.

\section{Methodology}
\label{sec:methodology}

In this section, we present the fundamental framework of our algorithm. The core idea is to construct a set of adaptive basis functions using single-hidden-layer neural networks (SLFN) and solve PDEs through a collocation scheme.
\Cref{sec:problem_state} introduces the linear PDE setting and the corresponding collocation framework, followed by an outline of the overall algorithmic procedure. Subsequently, \Cref{sec:basis_construction} and \ref{sec:adaptive} describe, respectively, the training process for constructing neural network basis functions and the adaptive strategies designed to enhance the solver performance.

\subsection{Problem Statement and Collocation Scheme}
\label{sec:problem_state}
Consider the following linear PDE defined on a bounded domain \(\Omega \subset \mathbb{R}^d\) with boundary \(\partial \Omega\):  
\begin{equation}
    \label{equ:linear_pde}
    \begin{aligned}
        \mathcal{L}u = f, \quad \text{in } \Omega, \quad 
        \text{and }
        \mathcal{B}u = g, \quad \text{on } \partial \Omega,
    \end{aligned}
\end{equation}
where \(\mathcal{L}\) is a linear differential operator, \(\mathcal{B}\) is a linear boundary operator, and \(f\) and \(g\) are the source term and prescribed boundary data, respectively.

For the collocation method, the PDE and boundary conditions are enforced at discrete points within the domain and on the boundary. In particular, let the collocation points for the domain and boundary be \(X_\Omega = \{\bm{x}_m\}_{m=1}^{M_\Omega}\) and \(X_{\partial \Omega} = \{\bm{x}_m'\}_{m=1}^{M_{\partial \Omega}}\), respectively. 
Given a set of basis functions \(\Phi = \{\phi_n\}_{n=1}^N\), and let the approximated solution given by $u_N(\bm{x}) = \sum_{n=1}^N \beta_n \phi_n(\bm{x})$, then the resulting least squares system is given by 
\[
    \bm{A} \bm{\beta}=\bm{b},
\]
where $\bm{\beta} = [\beta_1, \ldots, \beta_N]^T$ and 
\[
\bm{A}=\left(\begin{array}{c}
	\mathcal{L} \phi_n\left(\bm{x}_m\right)_{m, n=1}^{M_{\Omega}, N} \\[6pt]
	\lambda \phi_n\left(\bm{x}_m^\prime\right)_{m, n=1}^{M_{\partial \Omega}, N}
\end{array}\right)	
\in \mathbb{R}^{M\times N}, \quad 
\bm{b}=\left(\begin{array}{c}
	f\left(\bm{x}_m\right)_{m=1}^{M_{\Omega}} \\[6pt]
	\lambda g\left(\bm{x}_m^\prime\right)_{m=1}^{M_{\partial \Omega}}
\end{array}\right) \in \mathbb{R}^M.
\]
The upper block of \(\bm{A}\) corresponds to enforcing the PDE in the interior domain, while the lower block enforces the boundary conditions with penalty parameter \(\lambda\in\mathbb{R}^+\). We define this process as a function
\[
	\label{equ:least_squares}
	\bm{\beta}=\text{\scshape ColloLSQ}(\mathcal{L}, f,\mathcal{B}, g, \Phi, X_\Omega,X_{\partial\Omega},\lambda).
\]

For traditional numerical approaches, e.g., Fourier series or polynomial expansions, the choice of basis functions is typically prescribed. While effective in many scenarios, these methods may become inadequate when trying to resolve localized features or steep gradients. To overcome these limitations,
we propose to construct adaptive basis functions using neural networks by employing a greedy procedure, which allows more flexibility in capturing complex solution behaviors directly from the residual of the PDEs.

To facilitate a global understanding of the proposed methodology, Algorithm \ref{alg:entire} presents the overall PDE-solving framework, referencing key sub-algorithms that are elaborated upon in the subsequent subsections. This structure provides the reader with a clear roadmap for understanding how each component integrates into the complete solution process.
\begin{algorithm}[hbpt]
    \caption{Main Algorithm}
    \label{alg:entire}
    \SetAlgoLined
    \SetAlgoNoEnd
    Start with an initial guess \(u_0\) and set the initial basis \(\psi_0 = u_0\). Generate uniform collocation points \(X_\Omega^1\) and \(X_{\partial\Omega}\), and set \(s=1\)\;
    Sampling new residual-based points \(X_{\Omega}^{s,2}\) by \emph{Rejection Sampling} (see Algorithm \ref{alg:rejection}), set \(X_\Omega^s = X_{\Omega}^1 \cup X_{\Omega}^{2,s}, X_{\partial\Omega}^s = X_{\partial\Omega}\)\;
    Constructing a new neural network basis \(\psi_s\) by \emph{Adaptive Initialization} (see Algorithm \ref{alg:adaptive_init})\;
    Training the basis \(\psi_s\) to approximate the residual equation \eqref{equ:residual} using  \emph{Adaptive Basis Training} (see Algorithm \ref{alg:basis_training})\;
    Updating the approximation \(u_s\) by {\scshape ColloLSQ} in the augmented space \(\text{span}\{\psi_i\}_{i=0}^s\) (omit \(\psi_0\) if it equals zero)\;
    Set \(s\gets s+1\) and repeat steps (2)-(6) until the desired accuracy is achieved.
\end{algorithm}

\subsection{Construction of Neural Network Basis}
\label{sec:basis_construction}
We construct each basis function using a single-hidden-layer feedforward neural network (SLFN), which allows for efficient training using a hybrid strategy detailed in this section.

As demonstrated in the Extreme Learning Machine (ELM) framework \cite{huang2006extreme}, SLFNs can be trained efficiently by fixing the weights and biases of the hidden layer and solving a linear least squares problem to determine the output layer coefficients. While this approach reduces the optimization complexity, it may limit the expressiveness of the network. To enhance the representational capacity of the network, we use an \emph{adaptive initialization} strategy based on the residual of the PDE and apply a small number of Adam optimization steps after solving the least squares problem to refine the hidden layer parameters.

An SLFN with \( N \) hidden neurons is mathematically expressed as:
\[
	f_{NN}(\bm{x}) = \sum_{i=1}^N c_i\,\sigma\big(\bm{w}_i \cdot \bm{x} + b_i\big),\quad \bm{x} \in \Omega \subset \mathbb{R}^d,
\]
where \( c_i \in \mathbb{R} \) are the output layer coefficients, \( \sigma \) is the activation function, and \( \bm{w}_i \in \mathbb{R}^d \), \( b_i \in \mathbb{R} \) represent the weights and biases of the hidden neurons. We define \( \bm{W} = [\bm{w}_1, \bm{w}_2, \ldots, \bm{w}_N] \), \( \bm{b} = (b_1, b_2, \ldots, b_N) \), and \( \bm{c} = (c_1, c_2,\cdots, c_N) \), encapsulating the network parameters as \( \theta = \{\bm{W}, \bm{b}, \bm{c}\} \).

Starting with an initial approximation \( u_0 \) (typically zero) and the initial basis function \(\psi_0 := u_0\), we iteratively construct adaptive basis functions using SLFNs. Following the multistage training strategy outlined in \cite{wang_multi-stage_2024}, we refer to the construction of each basis function as a stage. At stage \( s > 0 \), given the current approximation \( u_{s-1} \) and basis functions \( \Psi_{s-1} = \{\psi_i\}_{i=0}^{s-1} \) (omit \(\psi_0\) if it equals zero), we initialize a new SLFN basis \(\psi_s\) with width \(N_s\) and parameters \(\theta_s = \{\bm{W}_s, \bm{b}_s, \bm{c}_s\}\). The hidden layer parameters \((\bm{W}_s, \bm{b}_s)\) are adaptively initialized (detailed in \Cref{sec:adaptive}) to ensure flexibility in capturing the key features of the residual, which is defined as:
\begin{equation}
	\label{equ:residual}
	\begin{aligned}
		\mathcal{L}e_{s-1}(\bm{x}) &= f_{s-1}^{res}(\bm{x}) := f(\bm{x}) - \mathcal{L}u_{s-1}(\bm{x}), \quad \text{in } \Omega, \\ 
		\mathcal{B}e_{s-1}(\bm{x}) &= g_{s-1}^{res}(\bm{x}) := g(\bm{x}) - \mathcal{B}u_{s-1}(\bm{x}), \quad \text{on } \partial \Omega,
	\end{aligned}
\end{equation}
where \( e_s(\bm{x}) = u(\bm{x}) - u_s(\bm{x}) \) represents the error after stage \( s \).

To train \(\psi_s(\bm{x})\), we begin by fixing the hidden layer weights and biases, treating the SLFN as a linear combination of activation functions \(\Phi_s = \{\phi_{s,i}(\bm{x}) = \sigma(\bm{w}_{s,i} \cdot \bm{x} + b_{s,i})\}_{i=1}^{N_s}\). Next, we select the collocation points \(X_\Omega^s\) and \(X_{\partial \Omega}^s\) (as detailed in \Cref{sec:adaptive}) and apply the {\scshape ColloLSQ} function to compute the output layer coefficients \(\bm{c}_s\). The hidden layer parameters are then refined using the Adam optimizer to enhance the network performance, with the loss function defined as:
\begin{equation}
	\label{equ:loss}
	L(\theta_s) = \sum_{\bm{x} \in X} \left| \mathcal{L} \psi_s(\bm{x}) - f_{s-1}^{res}(\bm{x}) \right|^2 + 
	\lambda^2 \sum_{\bm{x} \in X_{bc}} \left| \mathcal{B} \psi_s(\bm{x}) - g_{s-1}^{res}(\bm{x}) \right|^2.
\end{equation}
Once the optimization process is completed, the output layer coefficients \(\bm{c}_s\) are updated once again using {\scshape ColloLSQ}. The complete training procedure is provided in Algorithm \ref{alg:basis_training}.

\begin{algorithm}[htbp]
	\caption{Adaptive Basis Training}
	\label{alg:basis_training}
	\SetAlgoLined
	\SetAlgoNoEnd
	\KwData{Residual equation information \((\mathcal{L},f_{s-1}^{res},\mathcal{B},g_{s-1}^{res})\) from \eqref{equ:residual}, current approximation \(u_{s-1}\), new basis \(\psi_s\) with parameters \(\theta_s=(\bm{W}_s, \bm{b}_s,\bm{c}_s)\), collocation points \(X_\Omega^s, X_{\partial\Omega}^s\), regularization parameter \(\lambda\), maximum optimization steps \(N_{opt}\), learning rate \(\alpha\)}
	\KwResult{Updated basis function \(\psi_{s}(\bm{x})\)}
	
	Compute the output layer coefficients \(\bm{c}_s = \text{\scshape ColloLSQ}(\mathcal{L}, f_{s-1}^{res},\mathcal{B},g_{s-1}^{res}, \Phi_s, X_\Omega^s, X_{\partial\Omega}^s, \lambda)\)\;
	\For{$n=1$ \KwTo $N_{opt}$}{
		Calculate the loss function $L(\theta_s)$ as in \eqref{equ:loss}\;
		Update the \(\theta_s\) by Adam with loss $L(\theta_s)$ and learning rate $\alpha$\;
	}
	Update the output layer coefficients $\bm{c}_s = \text{\scshape ColloLSQ}(\mathcal{L}, f_{s-1}^{res},\mathcal{B},g_{s-1}^{res}, \Phi_s, X_\Omega^s, X_{\partial\Omega}^s, \lambda)$\;
	\Return $\psi_{s}$
\end{algorithm}

\subsection{Adaptive Strategies}
\label{sec:adaptive}

As discussed in the previous section, SLFNs can achieve efficient training with low optimization errors. However, this efficiency often comes at the cost of reduced network expressiveness. Therefore, it becomes essential to adaptively initialize the hidden layer parameters to effectively capture the critical features of the target function. Moreover, the selection of training points plays an important role in maintaining both the efficiency and robustness of the training process.
In this section, we present adaptive strategies for both collocation point selection and parameter initialization to enhance the network performance.

\subsubsection{Adaptive Collocations}  
For smooth target functions with minimal variation, equidistant or uniform collocation is generally sufficient to achieve reliable approximations. However, in cases where the solution exhibits sharp gradients, singularities, or other localized features, adaptive sampling techniques based on the equation residual or gradient typically provide improved accuracy and stability.

In this work, we divide the internal collocation points into two sets at each stage \( s \geq 1 \). Two-thirds of these points, denoted as \( X_\Omega^{1} \), are sampled equidistantly to ensure comprehensive coverage of the domain. The remaining one-third, denoted as \( X_\Omega^{2,s} \), are selected adaptively through a rejection sampling process guided by the equation residual \( f_{s-1}^{res} \), as described in Algorithm \ref{alg:rejection}. By assigning a higher likelihood of selection to points with larger residuals, we focus on refining areas most in need. Boundary collocation points are sampled equidistantly throughout to ensure consistent enforcement of boundary conditions.

\begin{algorithm}[htbp]
	\caption{Rejection Sampling}  
	\label{alg:rejection}  
	\SetAlgoLined  
	\SetAlgoNoEnd  
	\KwData{Sampling distribution function \(D(\bm{x})\), number of points \(N\)}  
	\KwResult{Adaptive collocation points \(X\)}  

	Set \(M = \max_{\bm{x} \in \Omega} |D(\bm{x})|\)\;  
	\(X \gets \emptyset\)\;  
	\While{\(|X| < N\)}{  
		Sample a candidate point \(\bm{x}\) uniformly from \(\Omega\)\;  
		Sample a random number \(r\) uniformly from \([0, M]\)\;  
		\If{\(|D(\bm{x})| \geq r\)}{  
			\(X \gets X \cup \{\bm{x}\}\)\;  
		}  
	}  
	\Return \(X\)
\end{algorithm}

To handle sampling in complex geometrical domains, we employ an indicator function to determine whether a candidate point lies within the domain:  
\[
	\chi_\Omega(\bm{x}) = \begin{cases}
		1, & \bm{x} \in \Omega, \\  
		0, & \bm{x} \notin \Omega.  
	\end{cases} 
\]  
Initially, candidate points are uniformly generated within the bounding box that encloses the domain. The indicator function is then used in combination with rejection sampling to filter these points, ensuring that only those within the domain are retained. This approach allows for efficient handling of irregularly shaped or disconnected domains without requiring explicit parameterization of the boundaries. 



\subsubsection{Adaptive Initialization}
\label{subsec:ada_init}
The initialization of network parameters is crucial for the convergence and accuracy of the solution. Both the weights and the biases determine the underlying approximating subspace of the SLFN. To enhance the network capacity to capture complex features, we propose an adaptive initialization strategy that leverages the equation residual to guide the selection of hidden layer parameters.

To achieve effective initialization, it is essential to understand how the parameters influence the network capacity. For an SLFN with fixed hidden layer parameters \((\bm{W},\bm{b})\), its approximation subspace consists of functions of the form:
\[
\phi_i(\bm{x}) = \sigma(\bm{w}_i \cdot \bm{x} + b_i), \quad i=1,2,\ldots,N.
\]
For the most commonly used activation functions, such as sigmoid, tanh, Gaussian, ReLU, etc., the weight \(\bm{w}_i\) controls the slope or frequency of \(\phi_i\), and \(\phi_i\) exhibits the most significant feature around the hyperplane:
\[
    \{\bm{x} \in \mathbb{R}^d : \bm{w}_i \cdot \bm{x} + b_i = 0\}.
\]
We refer to this hyperplane as the partition hyperplane of \(\phi_i\) following \cite{zhang_transferable_2024}. Both the frequency and location of the hyperplane are crucial for capturing the features of the target function.

Since weights are the only parameters that control the frequency of basis functions, we first consider the initialization of weights. Neural networks are known to exhibit spectral bias \cite{rahaman_spectral_2019}, or the frequency principle \cite{xu_frequency_2020}. By leveraging the iterative approximation process, we initialize the weights at each stage \(s\) using a scaling factor \(R_s\) to emphasize different frequency components of the target function. The scaling factors \(\{R_s\}\) form a non-decreasing sequence, with weights initialized as \(\bm{W}_s \sim \mathbb{U}(-R_s, R_s)\). This strategy enhances the ability of networks to represent solutions with varying spectral characteristics, effectively addressing challenges in multiscale phenomena.

With the weights fixed, we determine the biases \(b_{s,i}\), which influence the partition hyperplane of \(\phi_{s,i}(\bm{x})\). As in \cite{zhang_transferable_2024}, we define the hyperplane density at \(\bm{x}\) as the percentage of neurons whose partition hyperplane intersects the ball with center \(\bm{x}\) and radius \(\tau\):
\[
    D_\tau(\bm{x};\bm{W},\bm{b}) = \frac{1}{N} \sum_{i=1}^N \mathbb{I}\left(\frac{|\bm{w}_i \cdot \bm{x} + b_i|}{\|\bm{w}_i\|_2} \leq \tau\right).
\]
It is reasonable to consider that if the hyperplanes pass through the points where the target function is largest in magnitude, the network will have good approximation capability. Thus, we employ the rejection sampling algorithm to select a set of base points \(X_{base}^s=\{\bm{x}_{s,i}\}_{i=1}^{N_s}\) based on the equation residual \(f^{res}_{s-1}\). We then initialize the biases \(\bm{b}_s=(b_{s,1},b_{s,2},\cdots,b_{s,N_s})\) such that the partition hyperplane of \(\phi_{s,i}\) passes through the base points:
\begin{equation}
    \label{equ:init_b}
    b_{s,i} = -\bm{w}_{s,i} \cdot \bm{x}_{s,i}, \quad i=1,2,\ldots,N_s.
\end{equation}
This initialization enhances the expressiveness of the neural network at key base points, thereby improving its ability to capture local features and increasing approximation accuracy in regions with large residuals. To illustrate the effectiveness of this approach, we use \(|\!\sin x \cos y|\) as the target distribution. Figure \ref{fig:hyperplane_density} presents the resulting density distributions of the hyperplanes initialized using \eqref{equ:init_b}. The results demonstrate that the proposed method achieves an appropriately concentrated distribution.  

\begin{figure}[htbp]
	\centering
	\includegraphics[width=.63\textwidth]{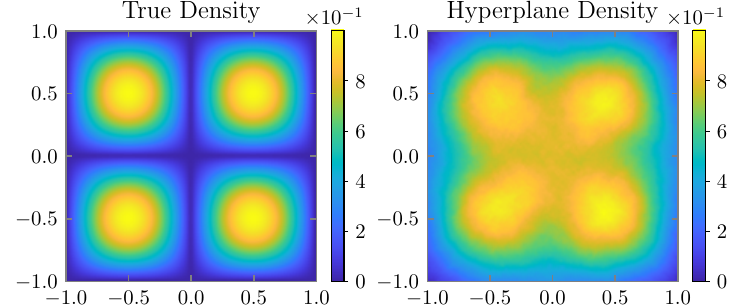}
	\caption{An example of biases adaptive initialization. Left: The target distribution \(|\!\sin x \cos y|\); Right: The density distribution of 20000 initialized hyperplanes computed with \(\tau = 0.1\).}
	\label{fig:hyperplane_density}
\end{figure}

In addition to parameter initialization, the selection of an appropriate network width \(N_s\) also plays a key role in ensuring algorithm convergence. Rather than discussing it here, we defer this topic to the next section, where the analysis provides essential guidance. A detailed strategy for choosing the network width can be found in \Cref{rem:width_choice}.
Finally, the complete adaptive initialization strategy is summarized in Algorithm \ref{alg:adaptive_init}.

\begin{algorithm}[htbp]
    \caption{Adaptive Initialization}
    \label{alg:adaptive_init}
    \SetAlgoLined
    \SetAlgoNoEnd
    \KwData{Equation residual \(f^{res}_{s-1}\), number of hidden neurons \(N_s\), scaling factor \(R_s\)}
    \KwResult{Hidden layer parameters \((\bm{W}_s,\bm{b}_s)\)}
    
    \SetKwProg{Fn}{Function}{:}{end} 
    \Fn{{\scshape AdapInit}\((f^{res}_{s-1}, N_s, R_s)\)}{
        Initialize the weights \(\bm{W}_s \sim \mathbb{U}\left(-R_s, R_s\right)\)\;
        Select base points \(X_{base}^s=\{\bm{x}_{s,i}\}_{i=1}^{N_s}\) by \emph{Rejection Sampling} (see Algorithm \ref{alg:rejection}) based on \(f^{res}_{s-1}\)\;
        Initialize the biases \(b_{s,i} = -\bm{w}_{s,i} \cdot \bm{x}_{s,i}\)\;
        \Return \((\bm{W}_s,\bm{b}_s)\)
    }
\end{algorithm}

\section{Numerical analysis}
\label{sec:analysis}

In this section, we present a convergence analysis for the proposed method. Specifically, we give the approximation error estimate. Furthermore, under appropriate optimization assumptions, we derive {\it a posteriori} error indicator to guide the termination of the iterative process.

To this end, we consider the linear boundary value problem \eqref{equ:linear_pde} and introduce the associated norm defined by 
\[
    \vertiii{u} = \left(\|\mathcal{L}u\|_{L^2(\Omega)}^2 + \|\lambda\mathcal{B}u\|_{L^2(\partial\Omega)}^2\right)^{\frac{1}{2}},
\]
where $\lambda \in \mathbb{R}^+$. 
Obviously, \(\vertiii{\cdot}\) is indeed a norm. The linearity property follows immediately. The positive definiteness is guaranteed by the existence and uniqueness of the PDE solution, while the triangle inequality is derived using the Cauchy-Schwarz inequality. Note that in the special case of fitting problems, which can be interpreted as PDEs with the identity operator \(\mathcal{L}\) and no boundary operator, the norm \(\vertiii{\cdot}\) is actually the \(L^2\) norm.

Next, we prove that the error at each stage decreases with respect to the norm \(\vertiii{\cdot}\). At stage \(s\), we project the exact solution \(u\) onto the subspace spanned by \(\Psi_s = \{\psi_i\}_{i=0}^s\) in the sense of the norm \(\vertiii{\cdot}\). Although our algorithm employs a discrete projection in the collocation sense, the analysis considers a continuous projection given by:
\[
    u_s^* = \underset{v\in\text{span}\Psi_s}{\arg\min} \vertiii{u-v} := P_{\Psi_s}u.
\]
The errors and the improvements between consecutive stages are defined as:
\begin{equation}
    \label{equ:error_variants}
    e_s^* = u - u_s^*, \quad v_s^* = u_{s}^* - u_{s-1}^*.
\end{equation}
By the properties of the projection, it is clear that \(\vertiii{e_{s+1}^*} \leq \vertiii{e_{s}^*}\). This suggests that the error cannot increase at each stage. A more detailed analysis shows that, except when the solutions at consecutive stages are identical, the error strictly decreases.  
The following proposition formally establishes this behavior.

\begin{proposition}
    \label{prop:converge}
    Consider the linear boundary value problem \eqref{equ:linear_pde} with a source term \(f \in L^2(\Omega)\) and boundary data \(g \in L^2(\partial\Omega)\). The error defined in \eqref{equ:error_variants} either strictly decreases at each stage \(s \in \mathbb{N}^+\), i.e., \(\vertiii{e_{s+1}^*} < \vertiii{e_{s}^*}\) or remains unchanged if and only if \(u_{s+1}^* = u_s^*\).
\end{proposition}

\begin{proof}
   Direct calculation gives 
    \[
        \label{equ:es_expand}
        \vertiii{e_s^*}^2 = \vertiii{e_{s+1}^*}^2 + \vertiii{v_{s+1}^*}^2 + 2\left(\langle \mathcal{L}e_{s+1}^*, \mathcal{L}v_{s+1}^*\rangle_\Omega + \lambda^2\langle\mathcal{B}e_{s+1}^*, \mathcal{B}v_{s+1}^*\rangle_{\partial\Omega}\right).
    \]
    We claim that
    \begin{equation}
        \label{equ:inequality}
        \langle \mathcal{L}e_{s+1}^*, \mathcal{L}v_{s+1}^*\rangle_{\Omega} + \lambda^2\langle\mathcal{B}e_{s+1}^*, \mathcal{B}v_{s+1}^*\rangle_{\partial\Omega} \geq 0.
    \end{equation}
    Assuming \eqref{equ:inequality} holds, it follows that
    \[
        \vertiii{e_s^*}^2 \geq \vertiii{e_{s+1}^*}^2 + \vertiii{v_{s+1}^*}^2 \geq \vertiii{e_{s+1}^*}^2.
    \]
    The equality holds if and only if \(\vertiii{v_{s+1}^*} = 0\), which implies \(v_{s+1}^* = u_{s+1}^* - u_{s}^* = 0\).

    It remains to prove \eqref{equ:inequality}. Suppose, for the sake of contradiction, that \eqref{equ:inequality} does not hold. Define \(v_{s+1}^\eta = (1-\eta)v_{s+1}^* \in \text{span}\Psi_{s+1}\). Then, we have
    \[
        \begin{aligned}
            \vertiii{e_s^*-v_{s+1}^\eta}^2 =& \vertiii{e_{s}^*-v_{s+1}^*}^2 + \eta\big(\eta\vertiii{v_{s+1}^*}^2 \\
            &+ 2\langle \mathcal{L}v_{s+1}^*,\mathcal{L}e_{s+1}^*\rangle_{\Omega}+ 2\lambda^2\langle \mathcal{B}v_{s+1}^*,\mathcal{B}e_{s+1}^*\rangle_{\partial\Omega}\big).
        \end{aligned}
    \]
    By choosing \(\eta\) such that
    \[
        0 < \eta < -\frac{2\left(\langle \mathcal{L}v_{s+1}^*, \mathcal{L}e_{s+1}^*\rangle_{\Omega} + \lambda^2\langle \mathcal{B}v_{s+1}^*, \mathcal{B}e_{s+1}^*\rangle_{\partial\Omega}\right)}{\vertiii{v_{s+1}^*}^2},
    \]
    we would obtain \(\vertiii{e_s^* - v_{s+1}^\eta}^2 < \vertiii{e_s^* - v_{s+1}^*}^2\), contradicting the definition of \(v_{s+1}^*\) as the minimizer 
    \begin{equation}
        \label{equ:v_def}
        \begin{aligned}
            v_{s+1}^* &= u_{s+1}^*-u_s^*
            =\underset{v\in\text{span}\Psi_{s+1}}{\arg\min} \vertiii{u-v}-u_s^*\\
            &=\underset{v^\prime\in\text{span}\Psi_{s+1}}{\arg\min} \vertiii{u-(u_s^*+v^\prime)}+u_s^*-u_s^*\\
            &=\underset{v^\prime\in\text{span}\Psi_{s+1}}{\arg\min} \vertiii{e_s^*-v^\prime}.
        \end{aligned}
    \end{equation}
    Therefore, \eqref{equ:inequality} must hold.
\end{proof}

\Cref{prop:converge} establishes the fundamental \emph{qualitative} convergence property of Algorithm \ref{alg:entire}. In fact, if we further assume that at the $s$-th stage, the trained neural network is sufficiently good to the residual $e_s$, we can derive a more delicate \emph{quantitative} error estimate regarding the approximation error. To this end, let us consider the space of SLFNs defined by 
\[
    V_N^\sigma = \left\{ f_{NN}:\mathbb{R}^d \to \mathbb{R} \ \middle|\ f_{NN}(\bm{x}) = \sum_{i=1}^N c_i\,\sigma\big(\bm{w}_i \cdot \bm{x} + b_i\big)\right\}.
\]
The universal approximation property of SLFN has been established in \cite{hornik_approximation_1991} and is stated as follows.

\begin{theorem}[Universal Approximation \cite{hornik_approximation_1991}]
    Suppose \(1 \leq p < \infty, 0 \leq s < \infty\), and \(\Omega \subset \mathbb{R}^d\) is compact. If \(\sigma \in C^s(\Omega)\) is non-constant and bounded, then the space 
    \[
    V^\sigma := \bigcup_{N=1}^\infty V_N^\sigma
    \]
    is dense in the Sobolev space
    \[
        W^{s, p}(\Omega) := \left\{v \in L^p(\Omega): D^\alpha v \in L^p(\Omega), \ \forall |\alpha| \leq s \right\}.
    \]
\end{theorem}

In each stage \(s>0\), for any \(\tau_s > 0\), the above universal approximation theorem indicates that there exists a network size \(n(\tau_s, u_{s-1}^*)\) such that for \(n \geq n(\tau_s, u_{s-1}^*)\), one can find \(\psi_{s} \in V_n^\sigma\) satisfying  
\begin{equation}
    \label{equ:approx}
    \frac{\vertiii{e_{s-1}^*-\psi_{s}}}{\vertiii{e_{s-1}^*}} \leq \tau_s.
\end{equation}
Assume that sufficient training has been performed such that the relative error bound \eqref{equ:approx} holds. Then we can establish the following proposition for the {\it a posteriori} error estimate.

\begin{proposition}
\label{prop:posteriori}
   At stage \(s>0\), given \(0<\tau_s<1\), if  \eqref{equ:approx} holds, then
    \begin{equation}
        \label{equ:posteriori}
        \frac{1}{1 + \tau_s} \vertiii{v_{s}^*} \leq \vertiii{u - u_{s-1}^*} \leq \frac{1}{1 - \tau_s} \vertiii{v_{s}^*}.
    \end{equation}
\end{proposition}

\begin{proof}
    Selecting \(v^\prime=\psi_{s}\in\Psi_{s}\) in \eqref{equ:v_def} and utilizing \eqref{equ:approx}, we obtain  
    \[
    \begin{aligned}
        \left| \vertiii{u - u_{s-1}^*} - \vertiii{v_{s}^*} \right| \leq \vertiii{u - u_{s-1}^* - v_{s}^*} \leq \vertiii{e_{s-1}^* - \psi_{s}} \leq \tau_s\vertiii{e_{s-1}^*}.
    \end{aligned}
    \]
    Thus, we obtain the following bounds:  
    \[
        \vertiii{u - u_{s-1}^*} \leq \vertiii{v_{s}^*} + \tau_s \vertiii{u - u_{s-1}^*},
    \]
    and  
    \[
        \vertiii{u - u_{s-1}^*} \geq \vertiii{v_{s}^*} - \tau_s \vertiii{u - u_{s-1}^*}.
    \]
    Combining these inequalities yields \eqref{equ:posteriori}.
\end{proof}

\Cref{prop:posteriori} demonstrates that \(\vertiii{v_s}\) serves as the {\it a posteriori} error estimator to guide the stopping criterion. Moreover, the proposed method allows for training a new network to refine the current approximate solution if \(\vertiii{v_s}\) exceeds the prescribed tolerance, without requiring the retraining of previous networks.

\begin{proposition}
    \label{prop:exp_converge}
    Let \(S \geq 1\). If \(\left\{\tau_s\right\}_{s=1}^S\) and \(\left\{n\left(\tau_s, u_{s-1}^*\right)\right\}_{s=1}^S\) are chosen such that \eqref{equ:approx} holds, then
    \begin{equation}
    \label{equ:converge}
    \vertiii{u-u_S^*}\leq \vertiii{u-u_0^*} \cdot \prod_{s=1}^{S} \min \left\{1,2 \tau_s\right\}.
    \end{equation}
\end{proposition}

\begin{proof}
    For each stage \(s>0\), it is obvious that $\vertiii{u-u_{s}^*}\leq \vertiii{u-u_{s-1}^*}$. In addition, we observe that \(u_{s-1}^*+\vertiii{u-u_{s-1}^*}/\vertiii{\psi_{s}}\cdot\psi_{s}\in\Psi_{s}\). According to the definition of \(u_{s}^*\), we have
    \[
    \label{equ:error_reduce}
    \begin{aligned}
        \vertiii{u-u_{s}^*}&\leq\vertiii{u-\left(u_{s-1}^*+\frac{\vertiii{u-u_{s-1}^*}}{\vertiii{\psi_{s}}}\cdot\psi_{s}\right)}\\
        &=\vertiii{u-u_{s-1}^*}\vertiii{\frac{u-u_{s-1}^*}{\vertiii{u-u_{s-1}^*}}-\frac{\psi_{s}}{\vertiii{\psi_{s}}}}\\
        &< 2\tau_{s}\vertiii{u-u_{s-1}^*}.
    \end{aligned}
    \]
    The final step follows thanks to
    \[
    \label{equ:error_assumption}
    \begin{aligned}
        \vertiii{\frac{e_{s-1}^*}{\vertiii{e_{s-1}^*}} - \frac{\psi_{s}}{\vertiii{\psi_{s}}}} &\leq \vertiii{\frac{e_{s-1}^*}{\vertiii{e_{s-1}^*}} - \frac{\psi_{s}}{\vertiii{e_{s-1}^*}}} + \vertiii{\frac{\psi_{s}}{\vertiii{e_{s-1}^*}} - \frac{\psi_{s}}{\vertiii{\psi_{s}}}}\\
        &=\frac{\vertiii{e_{s-1}^*-\psi_{s}}}{\vertiii{e_{s-1}^*}} + \frac{\left|\vertiii{\psi_{s}}-\vertiii{e_{s-1}^*}\right|}{\vertiii{e_{s-1}^*}\vertiii{\psi_{s}}}\cdot\vertiii{\psi_{s}}
        \leq 2\tau_s.
    \end{aligned}
    \]
    Thus, we have 
    $$\vertiii{u-u_s^*}\leq \vertiii{u-u_{s-1}^*} \cdot \min \left\{1,2 \tau_s\right\}. $$
    Then, the estimate \eqref{equ:converge} follows.
\end{proof}

\Cref{prop:exp_converge} highlights the convergence behavior of Algorithm \ref{alg:entire}. Assuming that \eqref{equ:approx} holds, the algorithm demonstrates favorable convergence characteristics. Furthermore, if \(\tau_s=\tau<1/2\) for all \(s\), the algorithm achieves geometric convergence.

\begin{remark}
    \label{rem:width_choice}
    The selection of the network width plays an important role in balancing convergence speed and computational efficiency. Due to its iterative nature, Algorithm \ref{alg:entire} allows flexibility to vary the network width \(N_s\) at each stage. Although in theory a fixed width \(N_s \equiv N\) can still achieve convergence by performing sufficient stages~\cite[Prop 2.12]{ainsworth_galerkin_2021}, it actually results in slow convergence in practice. The primary reason is that a fixed-width network has limited capacity to capture the higher-resolution features of the residual.  
    \Cref{prop:exp_converge} further suggests how this stagnation can be avoided. Observe that if \(\tau_s = \tau < 1/2\) for all \(s\), then the sequence \(\left\{n(\tau, u_{s-1}^*)\right\}_{s=1}^S\) is expected to be non-decreasing. In particular, the error for \(u_s^*\) is reduced by a factor of \(2 \tau\) compared to the error for \(u_{s-1}^*\). To ensure that \(\psi_s\) can capture the higher-resolution features of \(u - u_{s-1}^*\) (as compared to the relatively lower-resolution features of \(u - u_{s-2}^*\)), it is necessary that \(n(\tau, u_{s-1}^*) > n(\tau, u_{s-2}^*)\). Therefore, to achieve geometric convergence, the network width \(N_s\) must be increased progressively as \(s\) grows. In the experiments, following the setting in \cite{ainsworth_galerkin_2021}, we set the network width as \(N_s = N_0 \times 2^s\), where \(N_0\in \mathbb{Z}^+\) is a predefined width.
\end{remark}

\section{Numerical Results}
\label{sec:numerical-results}

We apply the proposed solver to several different classes of problems in one and two dimensions, demonstrating the effectiveness of our algorithm. If not mentioned, the prescribed parameters are listed in the Appendix.

For all subsequent examples, if not specified, we set \(u_0 = 0\) and use the \(\tanh\) activation function. The error metrics used in the numerical experiments are the \(L^\infty\) and \(L^2\) norms, both of which are approximated on fine equidistant grids. Specifically, a grid of 1000 points is used for the 1D case, while a \(300^2\) grid is employed for the 2D case. Let the set of these grid points be denoted by \(X\). The errors are then computed as follows:
\[
\begin{aligned}
    E_{L^\infty} &= \max_{\bm{x} \in X} |u(\bm{x}) - u_S(\bm{x})|, \\  
    E_{L^2} &= \sqrt{\frac{|\Omega|}{|X|} \sum_{\bm{x} \in X} |u(\bm{x}) - u_S(\bm{x})|^2},  
\end{aligned}
\]  
where \(u(\bm{x})\) is the exact solution and \(u_S(\bm{x})\) is the numerical solution.  

For circular domains, the uniform grid is generated in polar coordinates. The \(L^\infty\) error calculation remains unchanged, whereas the \(L^2\) error formula is adjusted to account for the polar grid distribution:
\[
E_{L^2} = \sqrt{\frac{P(\Omega)}{|X|} \sum_{\bm{x} \in X} \|\bm{x}\||u(\bm{x}) - u_S(\bm{x})|^2},  
\]  
where \(P(\Omega)\) denotes the perimeter of the circular domain.

\subsection{Function Fitting}
\label{subsec:function_fitting}

We begin by considering a function fitting problem. The target function $f$ is composed of the first four terms of the Fourier series for a square wave:
\[
    f(x) = \sin(x) + \frac{1}{3} \sin(3 \pi x) + \frac{1}{5} \sin(5 \pi x) + \frac{1}{7} \sin(7 \pi x).
    \label{equ:funcfit}
\]

We present the true and {\it a posteriori} errors of stage 2,4,6, along with the convergence result of $L^2$ and $L^\infty$ errors with respect to the number of stages in \Cref{fig:funcfit}. Initially, the low-frequency components of the error are learned, with subsequent stages focusing on the high-frequency error components. The error curves demonstrate that the proposed method has exponential convergence, achieving an $L^\infty$ error of 1.689e{-11} and an $L^2$ error of 3.652e{-12} after seven stages.
\begin{figure}[htbp]
    \centering
    \includegraphics[width=\textwidth]{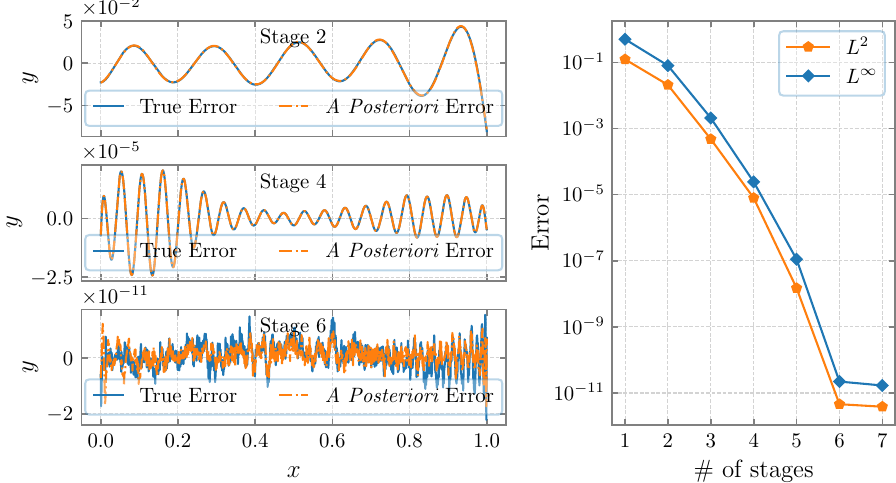}
    \caption{Results for Subsection \ref{subsec:function_fitting}. Left: the true and {\it a posteriori} errors of stage 2, 4, and 6. Right: Error vs. \# of stages.}
    \label{fig:funcfit}
\end{figure}

\subsection{Advection-diffusion problem with Boundary Layer}  
\label{subsec:boundary-layer}  

We now consider the following one-dimensional convection-diffusion equation with a boundary layer:  
\[  
\begin{aligned}  
-\varepsilon u^{\prime \prime} + b u^{\prime} &= f, \quad x \in \Omega = (-1, 1), \\  
u &= 0, \quad x \in \partial \Omega.  
\end{aligned}  
\]  
The solution to this problem consists of a smooth component and a boundary layer term of the form \(e^{\frac{b}{\varepsilon}(x-1)}\). For instance, when \(f \equiv 1\), the exact solution is given by
\[  
u(x) = \frac{x}{b} + \frac{1}{b} \left( \frac{e^{-\frac{2b}{\varepsilon}} + 1}{e^{-\frac{2b}{\varepsilon}} - 1} \right) \left( \frac{2e^{\frac{b}{\varepsilon}(x-1)}}{e^{-\frac{2b}{\varepsilon}} + 1} - 1 \right).  
\]  
This example focuses on the difficulty in accurately resolving the boundary layer, particularly for small values of \(\varepsilon\). For numerical testing, we set \(\varepsilon = 10^{-2}\) and \(b = -1\).  

\Cref{fig:1DBoundary_layer_approx_error} presents the approximate solution and the absolute error for the final stage. 
\begin{figure}[htbp]
    \centering
    \includegraphics[width=.9\textwidth]{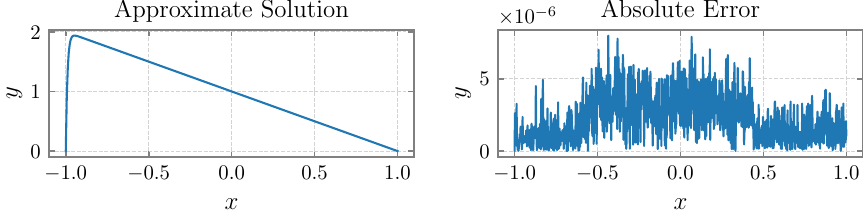}
    \caption{Approximation solution and the corresponding absolute error for Subsection \ref{subsec:boundary-layer}.}
    \label{fig:1DBoundary_layer_approx_error}
\end{figure}
\Cref{fig:1DBoundary_layer} shows the training result of stages \(s=2,3,4\) and \(L^2, L^\infty\) errors at each stage of the iteration. From the error curves, we observe that the boundary layer is initially not well captured. However, with successive corrections, our method successfully captures the boundary layer with high accuracy.

\begin{figure}[htbp]
    \centering
    \includegraphics[width=\textwidth]{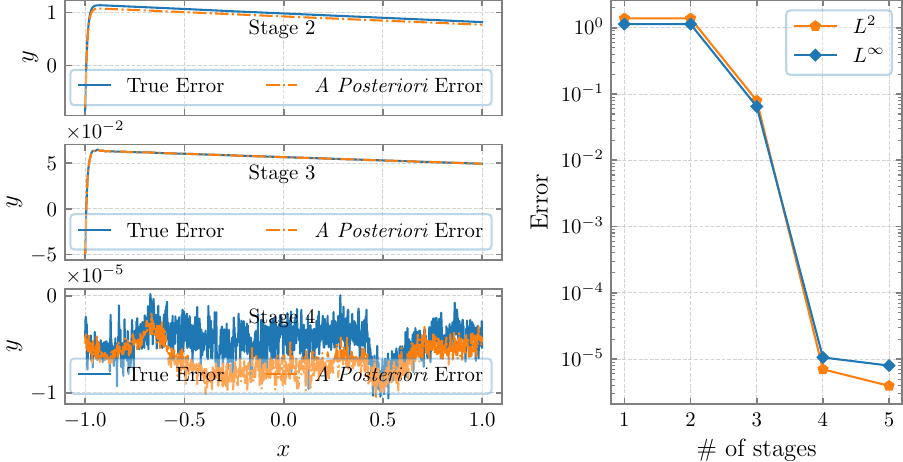}
    \caption{Results for Subsection \ref{subsec:boundary-layer}. Left: the true and {\it a posteriori} errors of stage 2,3, and 4. Right: Error vs. \# of stages.}
    \label{fig:1DBoundary_layer}
\end{figure}

\subsection{Poisson Equation}
The previous examples demonstrate the effectiveness of the algorithm. This subsection aims to show that the proposed method serves as a general framework, which can be further enhanced by integrating with other strategies to improve the solution of some challenging problems.
Considering the following Poisson equation
\begin{equation}
	\label{equ:poisson}
	\begin{aligned}
		-\Delta u &= f, &    \quad &\text{in } \Omega, \\
		u &= g,    & \quad &\text{on } \partial \Omega.
	\end{aligned}
\end{equation}
In this subsection, we analyze two challenging cases for solving \eqref{equ:poisson}: the first examines an L-shaped domain, where geometric singularities lead to non-smooth solutions; and the second focuses on a 2D problem where the source term exhibits rapid, localized variations, making it challenging to resolve near these regions. 

\subsubsection{L-shaped domain}
\label{subsubsec:L_shaped_domain}
We consider the Poisson equation \eqref{equ:poisson} on an L-shaped domain \(\Omega = (-1,1)^2 \setminus (-1,0]^2\) with a constant source term \(f \equiv 1\) and homogeneous Dirichlet boundary condition \(g \equiv 0\). The primary challenge in this problem arises from the presence of a nonconvex corner at the origin, where the solution exhibits singular behavior.

As established in \cite{ainsworth_extended_2024}, the solution of the Poisson equation in an L-shaped domain can be expressed as a series expansion that accounts for both singular and smooth components. Specifically, the solution is given by:
\[
u(x, y) = \sum_{i=1}^\infty c_i r^{\lambda_i} \sin \left( \lambda_i \theta \right) + u^*(x, y), \quad \lambda_i = \frac{2}{3} i,
\]
where \(u^* \in H^2(\Omega)\) represents the smooth part of the solution. The singular part, characterized by terms of the form \(r^{\lambda_i} \sin \left( \lambda_i \theta \right)\), captures the behavior near the nonconvex corner. 

In this experiment, we enhance the network by incorporating knowledge-based hidden neurons that correspond to the singular terms \(r^{\lambda_i} \sin \left( \lambda_i \theta \right)\) for \(i\) up to \(N_{k}\) in each training stage. This approach leverages the known structure of the solution to improve accuracy.

\Cref{fig:2DPoisson_Lshaped_extend} illustrates the absolute errors obtained without (left) and with (middle) the inclusion of knowledge-based hidden neurons. It can be observed that incorporating these neurons leads to improved accuracy, with a more consistent decay in the error curve. Notably, the knowledge-based hidden neurons do not introduce additional parameters for optimization; instead, they only slightly increase the size of the least squares system constructed by {\scshape ColloLSQ}. Consequently, the improvement in accuracy is achieved with a minimal increase in computational cost.
\begin{figure}[htbp]
	\centering
	\includegraphics[width=\textwidth]{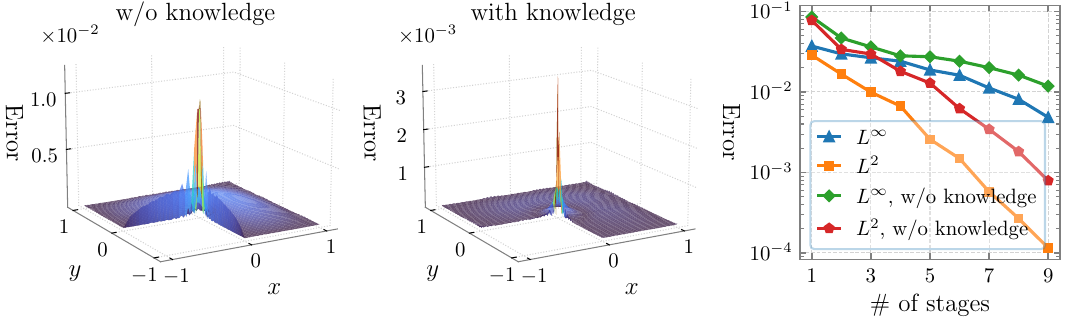}
	\caption{Poisson equation on a L-shaped domain. The absolute errors obtained without (left) and with (middle) the inclusion of knowledge-based hidden neurons. Right: \(L^2\) and \(L^\infty\) errors computed at each stage.}
	\label{fig:2DPoisson_Lshaped_extend}
\end{figure}
\Cref{fig:2DPoisson_Lshaped_extend_new_basis} presents the training results obtained using knowledge-based hidden neurons at stages \(s=3,5,7\), which include the true errors (left column), the {\it a posteriori} errors (middle column), and numerical solutions (right column).
\begin{figure}[htbp]
	\centering
	\includegraphics[width=\textwidth]{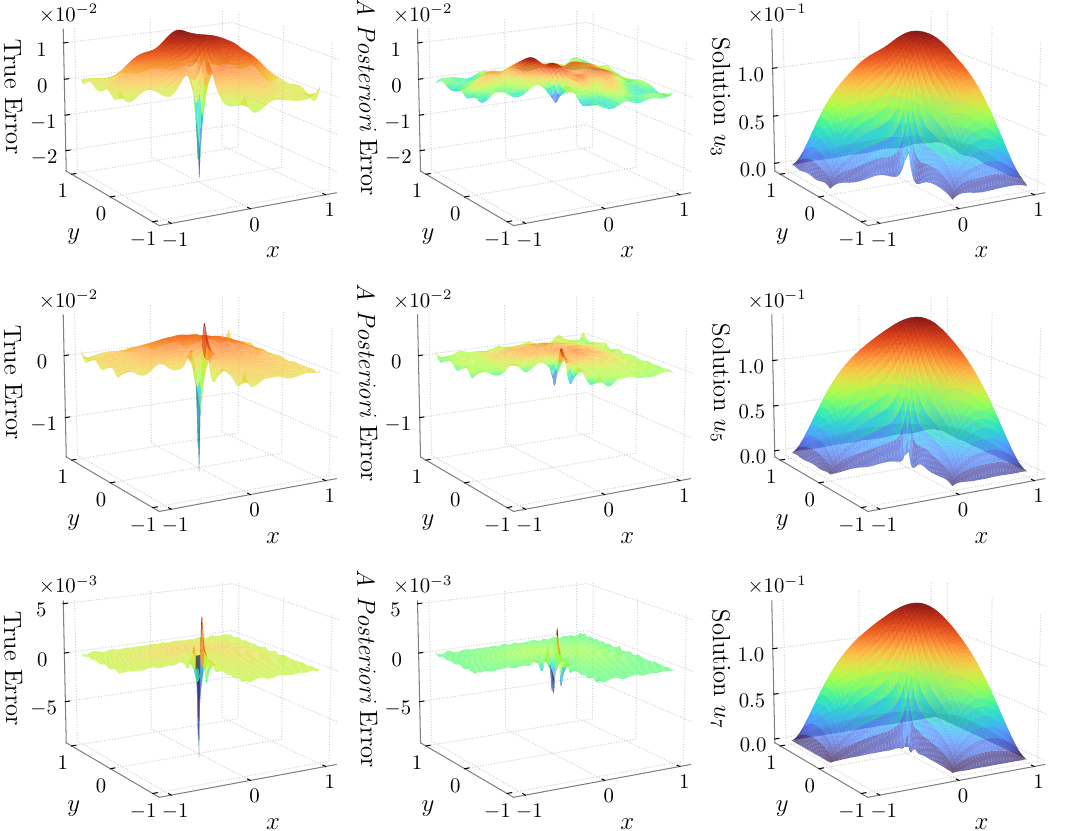}
	\caption{Poisson equation on a L-shaped domain with knowledge-based hidden neurons. True errors (left column), {\it a posteriori} errors (middle column), and numerical solution $u_s$ (right column) of stage 3, 5, and 7.}
	\label{fig:2DPoisson_Lshaped_extend_new_basis}
\end{figure}

\subsubsection{Rapidly Varying Source Term}
\label{subsubsec:rapidly_varying_source_term}

We study the Poisson equation \eqref{equ:poisson} over the domain \(\Omega = (0,1)^2\) with the exact solution given by
\begin{equation}
    u(\bm{x}) = 16\left(1 - x_1\right) x_1\left(1 - x_2\right) x_2 \left( \frac{1}{2} + \frac{1}{\pi} \arctan \left( \frac{r - \|\bm{x} - \bm{c}\|^2}{\varepsilon} \right) \right),
\end{equation}
where \(r = 1 / 16\) and \(\bm{c} = (1 / 2, 1 / 2)\), with the parameter \(\varepsilon\) controlling the sharpness of the transition region in the solution. We consider two values of \(\varepsilon\): \(1 / 120\) and \(1 / 500\). Smaller values of \(\varepsilon\) result in sharper gradients and more localized variations near the boundary of a ball centered at \(\bm{c}\) with radius \(\sqrt{r}\). The corresponding source terms for these two cases are shown in \Cref{fig:2DPoisson_sol_src}. Capturing such steep and localized variations in the source term is a challenging task. 

To address this issue, in addition to the (global) neuron functions using $tanh$ as the activate function, we introduce a kind of localized neuron function, inspired by the layer growth technique given in \cite{dang_adaptive_2024}. 
Given an approximate solution \(u^*\) and a corresponding set of points \(\{\bm{x}_i\}_{i=1}^{N_L}\), we construct a new network with \(N_L\) localized hidden neurons defined by
\[
    \phi_i(\bm{x}) := \exp\left( -\| \bm{k}_i \odot (\bm{x} - \bm{x}_i) \|^2 \right) \exp\left( -\frac{1}{2} \| \bm{k}_i \|^2 \left( u^*(\bm{x}) - u^*(\bm{x}_i) \right)^2 \right) 
\]
for $i = 1, \ldots, N_L$, where \(\bm{k}_i \sim \mathbb{U}(-R_L, R_L)\) is a shape parameter that modulates the localization effect, and \(\odot\) denotes element-wise multiplication. The localized neuron function \(\phi_i(\bm{x})\) is highly concentrated around the level set \(\{ \bm{x} : u^*(\bm{x}) = u^*(\bm{x}_i) \}\) and is further confined by a Gaussian function centered at \(\bm{x}_i\).
In this work, we determine the points \(\{\bm{x}_i\}_{i=1}^{N_L}\) using Algorithm \ref{alg:rejection}, with the sampling distribution function being the residual of the PDE, \(f - \mathcal{L}u^*\). 
We call these neurons adaptive localized neurons.
After initialization, the network is trained using the {\scshape ColloLSQ} method without further optimization steps.

In contrast to the adaptive initialization strategy discussed in \Cref{subsec:ada_init}, which focuses on capturing different frequency components and approximating the error distribution through globally defined neuron functions, the adaptive localized neuron approach constructs neuron functions that are highly localized around specific points, similar to radial basis functions. This localization allows the method to effectively capture the sharp and localized features of the source term.

\begin{figure}[htbp]
    \centering
    \includegraphics[width=0.75\textwidth]{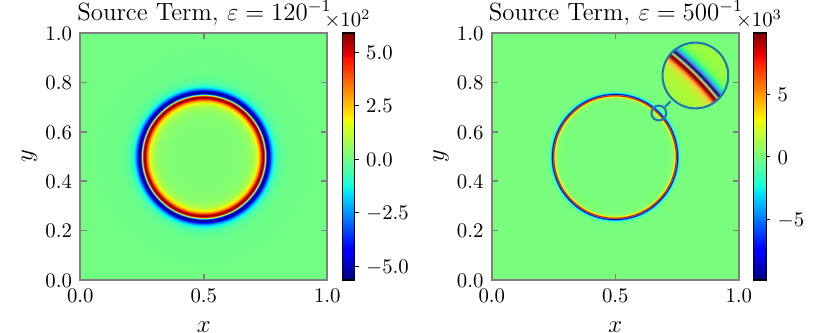}
    \caption{Poisson equation with a rapidly varying source term. The source term corresponds to \(\varepsilon = 120^{-1}\) (left) and \(\varepsilon = 500^{-1}\) (right), with a zoomed-in view.}
    \label{fig:2DPoisson_sol_src}
\end{figure}

The success of the adaptive approach hinges on the quality of the initial approximation \(u^*\). As such, the proposed method using $tanh$ as the activate function is first applied iteratively over \(S_1\) stages to obtain a sufficiently accurate initial solution. After that, the adaptive localized neuron strategy is employed over an additional \(S_2\) stages to further refine the solution. For consistency, we continue to denote the total number of stages as \(S\), i.e., \(S = S_1 + S_2\). 

The training results for both cases are presented in \Cref{fig:2DPoisson_result}. For \(\varepsilon=\) \(120^{-1}\), the method achieves an \(L^{\infty}\) error of 9.924e-7 and an \(L^2\) error of 1.969e-7, while for \(\varepsilon=500^{-1}\), the corresponding errors are 3.978e-6 and 8.951e-7. It is noteworthy that as \(\varepsilon\) decreases (i.e., as the source term becomes more localized), the error curve exhibits slower convergence and increased oscillations, indicating the greater difficulty in resolving sharper transitions. Nevertheless, by incorporating additional localized stages, we successfully reduce the errors in both cases to the same small scale. This highlights the efficiency of the adaptive strategy for achieving high accuracy.

\begin{figure}[htbp]
    \centering
    \includegraphics[width=\textwidth]{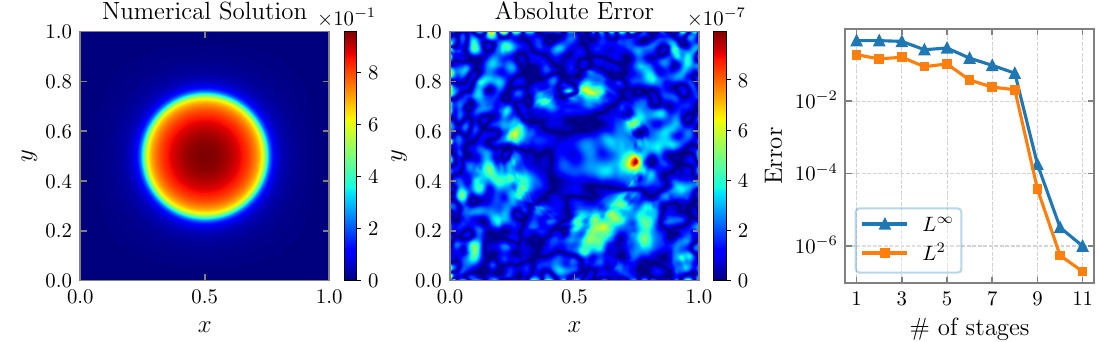}
    \includegraphics[width=\textwidth]{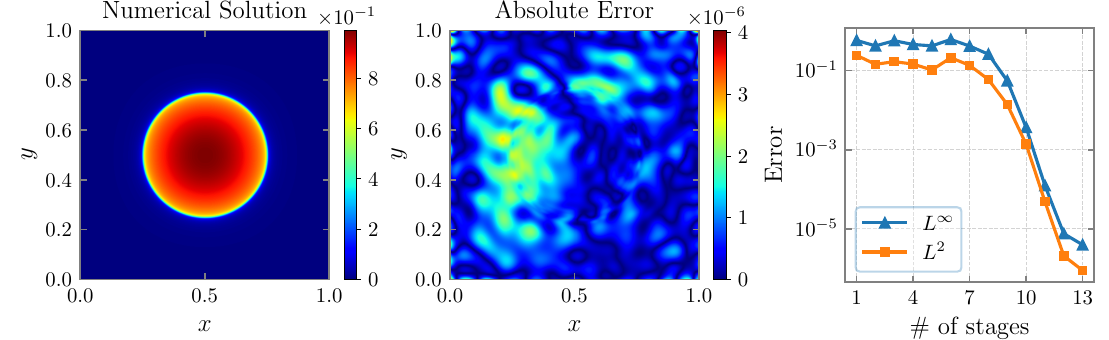}
    \caption{Poisson equation with a rapidly varying source term. Final-stage solution and error, with \(L^2\) and \(L^\infty\) errors at each stage. Top row: \(\varepsilon = 120^{-1}\); bottom row: \(\varepsilon = 500^{-1}\).}
    \label{fig:2DPoisson_result}
\end{figure}

\subsection{Fourth-order problem: Biharmonic equation}
\label{subsec:biharmonic}

In this subsection, we evaluate the proposed method for the biharmonic equation, a fourth-order partial differential equation, in both smooth and singular cases.

\subsubsection{Smooth solution}
\label{subsubsec:biharmonic-smooth}

Consider the biharmonic equation given by
\[
    \begin{aligned}
        & \Delta^2 u = f, && \bm{x} \in \Omega, \\
        & u = \partial_{n} u = 0, && \bm{x} \in \partial \Omega.
    \end{aligned}
\]
We test this on the domain \(\Omega = (-1,1)^2\) with a smooth exact solution
\[
    \begin{aligned}
        u(\bm{x}) = \sin^2 (\pi x) \sin^2 (\pi y) + \left(1 - x^2\right)^4 \left(1 - y^2\right)^4.
    \end{aligned}
\]

The training results are shown in \Cref{fig:2DBiharmonic_smooth}. The error curve demonstrates that the proposed method effectively handles the smooth case, with the error decreasing exponentially and achieving an $L^\infty$ error of 1.399e{-4} and an $L^2$ error of 5.526e{-5} after 7 stages.
\begin{figure}[htbp]
	\centering
	\includegraphics[width=\textwidth]{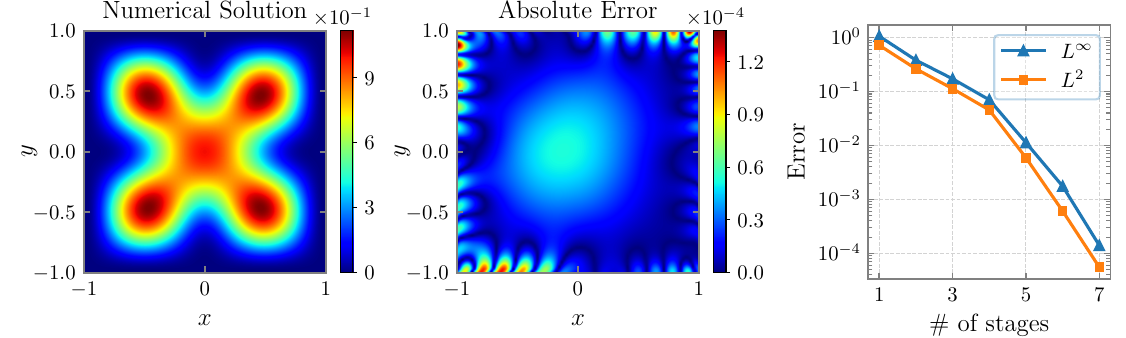}
	\caption{Biharmonic equation with smooth solution. Final-stage approximate solution and error, with \(L^2\) and \(L^\infty\) errors computed at each stage.}
	\label{fig:2DBiharmonic_smooth}
\end{figure}

\subsubsection{Unit circle with point load}
\label{subsubsec:biharmonic-singular}

We consider the biharmonic equation with a point load source term in the unit disk \(\Omega = \left\{(x, y): x^2 + y^2 < 1\right\}\):
\[
    \begin{aligned}
        \label{equ:biharmonic_a}\Delta^2 u &= \delta_{\{(0,0)\}}, &\qquad & \text{in } \Omega, \\
        u - \varepsilon_1 \partial_n(\Delta u) &= 0, &\quad & \text{on } \partial \Omega, \\
        \Delta u + \varepsilon_2 \partial_n u &= 0, &\quad & \text{on } \partial \Omega.
    \end{aligned}
\]
The Dirac delta function is defined in the sense of distributions. To avoid explicitly computing the Dirac function, we replace \eqref{equ:biharmonic_a} with the following equivalent governing equation.
\[
    2\pi r(\Delta u)_r = 1, \qquad \text{in } \Omega.
\]
Although the source term is no longer singular and the equation becomes third-order, the solution remains challenging to solve due to its low regularity, with a singular behavior in the second derivative at the origin, given by
\[
    u(r, \theta) = \frac{r^2}{8 \pi} \ln r + c_1 r^2 + c_2, \quad c_1 = \frac{1}{4 + 2 \varepsilon_2} \left(-\frac{1}{2 \pi} - \frac{\varepsilon_2}{8 \pi}\right), \quad c_2 = -c_1 + \frac{\varepsilon_1}{2 \pi}.
\]

\Cref{fig:2DBiharmonic_new_basis} shows the training results for stages \(s = 1, 3, 5\), including the true errors (left column), the {\it a posteriori} errors (middle column), and the approximate solutions (right column). From \Cref{fig:2DBiharmonic}, we observe that the logarithmic singularity at the origin slows down the error reduction; however, relatively good accuracy is still achieved, with an $L^\infty$ error of 1.028e-4 and an $L^2$ error of 3.304e-5.

\begin{figure}[htbp]
	\centering
	\includegraphics[width=\textwidth]{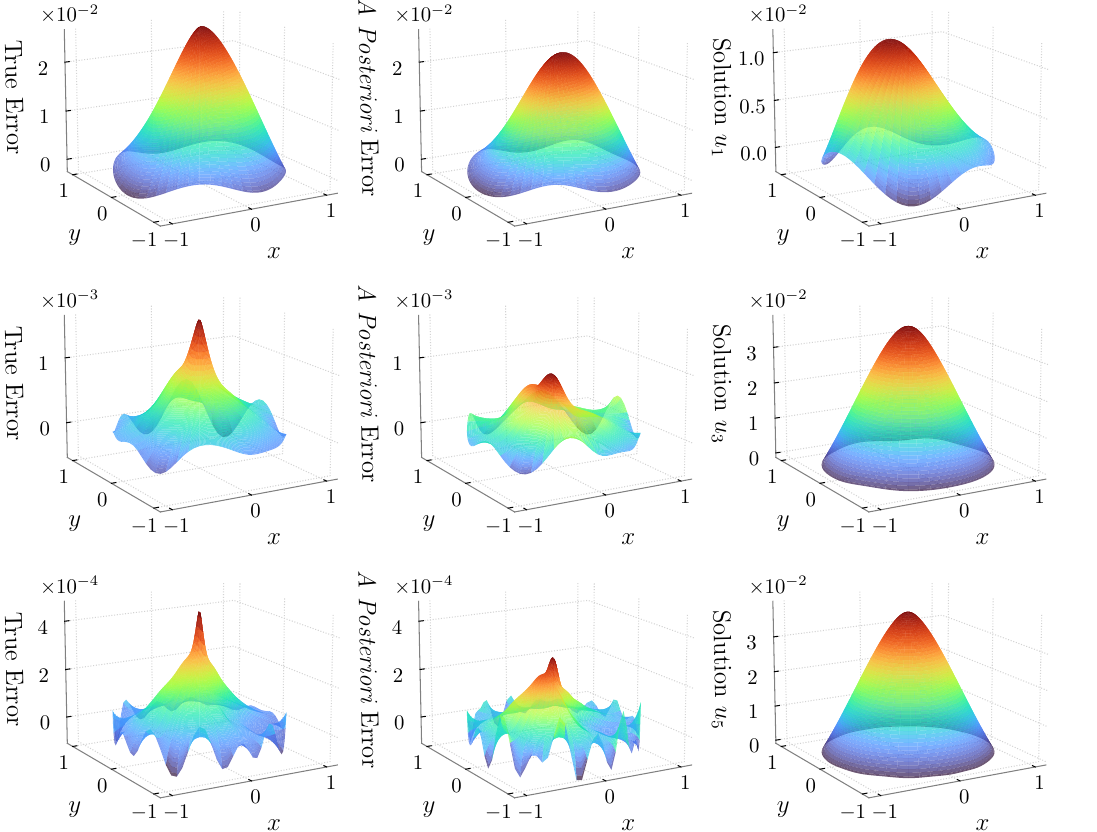}
	\caption{Biharmonic equation with point load. True errors (left column) and {\it a posteriori} errors (middle column) as well as numerical solutions $u_s$ (right column) for s = 1, 3, 5.}
	\label{fig:2DBiharmonic_new_basis}
\end{figure}

\begin{figure}[htbp]
	\centering
	\includegraphics[width=\textwidth]{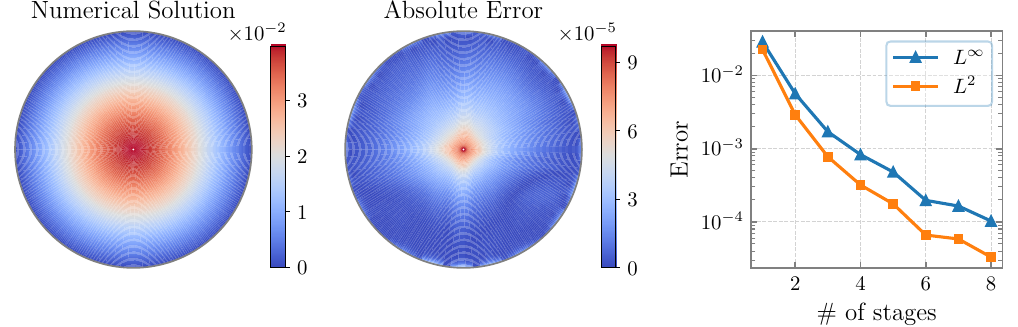}
	\caption{Biharmonic equation with point load. Final-stage approximate solution and error, with \(L^2\) and \(L^\infty\) errors computed at each stage.}
	\label{fig:2DBiharmonic}
\end{figure}


\subsection{A nonlinear example:  the Allen-Cahn equation}
\label{subsec:Allen_Cahn}
Note that the GNN is very time consuming, since it needs to compute the inner products with high accuracy. This prevents GNN from being used for nonlinear or time-dependent problems. Instead, the present collocation type method is more efficient in solving nonlinear or time-dependent problems. To this end, we consider the following nonlinear steady-state Allen-Cahn equation on \(\Omega = (-1, 1)^2\):  
\[
    \begin{aligned}
        -\varepsilon^2 \Delta u + u^3 - u &= 0  &   \quad &\text{in } \Omega, \\  
        u &= 0  &   \quad &\text{on } \partial \Omega.  
    \end{aligned}
\]  
To handle this nonlinear equation, we employ a fixed-point iteration approach. Specifically, at each iteration, the equation is linearized as follows:  
\begin{equation}
\label{equ:ac-linear}
\varepsilon^2 \Delta u_k - \alpha \varepsilon^2 u_k = u_{k-1}^3 - \left(1 + \alpha \varepsilon^2\right) u_{k-1},  
\end{equation}
where \(\alpha\) is a constant introduced to enhance iteration stability. Denote the linearized operator as \(\mathcal{L} = \varepsilon^2 \Delta - \alpha \varepsilon^2\), and the source term at each iteration as \(f(u_{k-1}) = u_{k-1}^3 - \left(1 + \alpha \varepsilon^2\right) u_{k-1}\).

Suppose after stage \(s\), we have an approximation \(u_s \in \text{span}\Psi_s\), where \(\Psi_s = \{ \psi_1, \ldots, \psi_s \}\) denotes the set of basis functions. The procedure for constructing the solution at stage \(s+1\) is described below:  
\begin{enumerate}[label=(\arabic*), left=0pt]  
    \item Initialize with \(i = 0\), set \(\tilde{u}_{s+1}^0 = u_s\), and construct a new SLFN basis \(\psi_{s+1}^{i+1}\) using adaptive initialization (Algorithm \ref{alg:adaptive_init}) based on the residual \(f^{\text{res}} = f(u_s) - \mathcal{L} u_s\).  
    \item Generate new residual-based internal collocations using rejection sampling (Algorithm \ref{alg:rejection}).  
    \item Solve \eqref{equ:ac-linear} for the approximation \(u_s^i \in \text{span}\Psi_s\) using {\scshape ColloLSQ} with the source term \(f(\tilde{u}_{s+1}^i)\).  
    \item Train the new basis function \(\psi_{s+1}^{i+1}\) to approximate the residual equation of \eqref{equ:ac-linear} by Algorithm \ref{alg:basis_training} with \(f^{\text{res}} = f(\tilde{u}_{s+1}^i) - \mathcal{L} u_s^i\) and \(g^{\text{res}} = -u_s^i\).  
    \item Update the solution \(\tilde{u}_{s+1}^{i+1} \in \text{span}(\Psi_s \cup \{ \psi_{s+1}^{i+1} \})\) using {\scshape ColloLSQ} with source term \(f(\tilde{u}_{s+1}^i)\).  
    \item Set \(i\gets i+1\) and repeat steps (2)-(6) until a predefined stopping criterion is met.  
    \item Suppose that the process terminates after \(I_{s+1}\) iterations, then update solution and basis set as
    \[
        u_{s+1} = \tilde{u}_{s+1}^{I_{s+1}}, \quad \Psi_{s+1} = \Psi_s\cup\{{\psi}_{s+1}^{I_{s+1}}\}.
    \]
\end{enumerate}  

Regarding the stopping criterion in Step (6), for simplicity in the experiments, we predefine the number of nonlinear iterations for each stage to prevent excessive iterations from causing error saturation. Of course, more adaptive criteria could be designed to better control the iterative process.

For the numerical experiment, we set \(\varepsilon = 0.1\) and \(\alpha \varepsilon^2 = 2\), and use the following initial guess
\[
	u_0(r) = 
	\begin{cases} 
	1, & \text{if } r \leq R_0, \\
	\frac{\log(R_1 / r)}{\log(R_1 / R_0)}, & \text{if } R_0 < r < R_1, \\
	0, & \text{if } r \geq R_1,
	\end{cases}
\]
where $R_0 = 0.7$, $R_1 = 0.9$. 

\Cref{fig:2DAllenCahn_result} shows the approximate solution and error at the final stage, along with the \(L^2\) and \(L^\infty\) errors at each nonlinear iteration. The error curves clearly indicate that the proposed method progressively improves the solution accuracy, achieving an \(L^\infty\) error of 2.350e-5 and an \(L^2\) error of 7.302e-6 after six stages and a total of 27 nonlinear iterations. \Cref{fig:2DAllenCahn_new_basis} further illustrates the true errors (left column), the {\it a posteriori} errors (middle column), and the numerical solutions \(u_s\) (right column) for stages \(s = 1,3,5\).  

\begin{figure}[htbp]
	\centering
	\includegraphics[width=\textwidth]{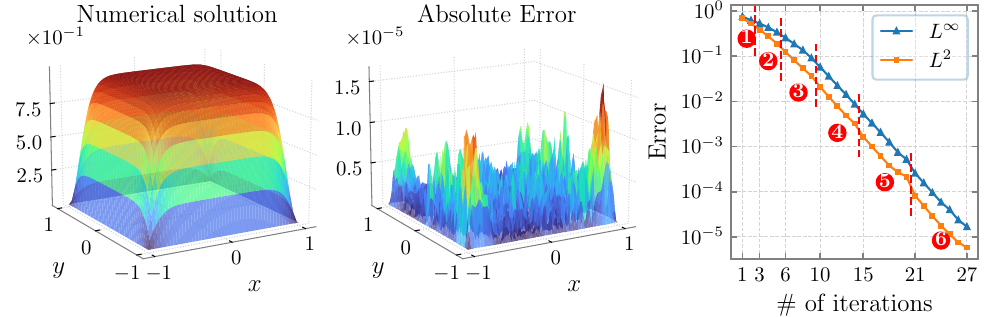}
	\caption{Steady-state Allen-Cahn equation. Final-stage approximate solution and error, with \(L^2\) and \(L^\infty\) errors computed at each nonlinear iteration.}
	\label{fig:2DAllenCahn_result}
\end{figure}

\begin{figure}[htbp]
	\centering
	\includegraphics[width=\textwidth]{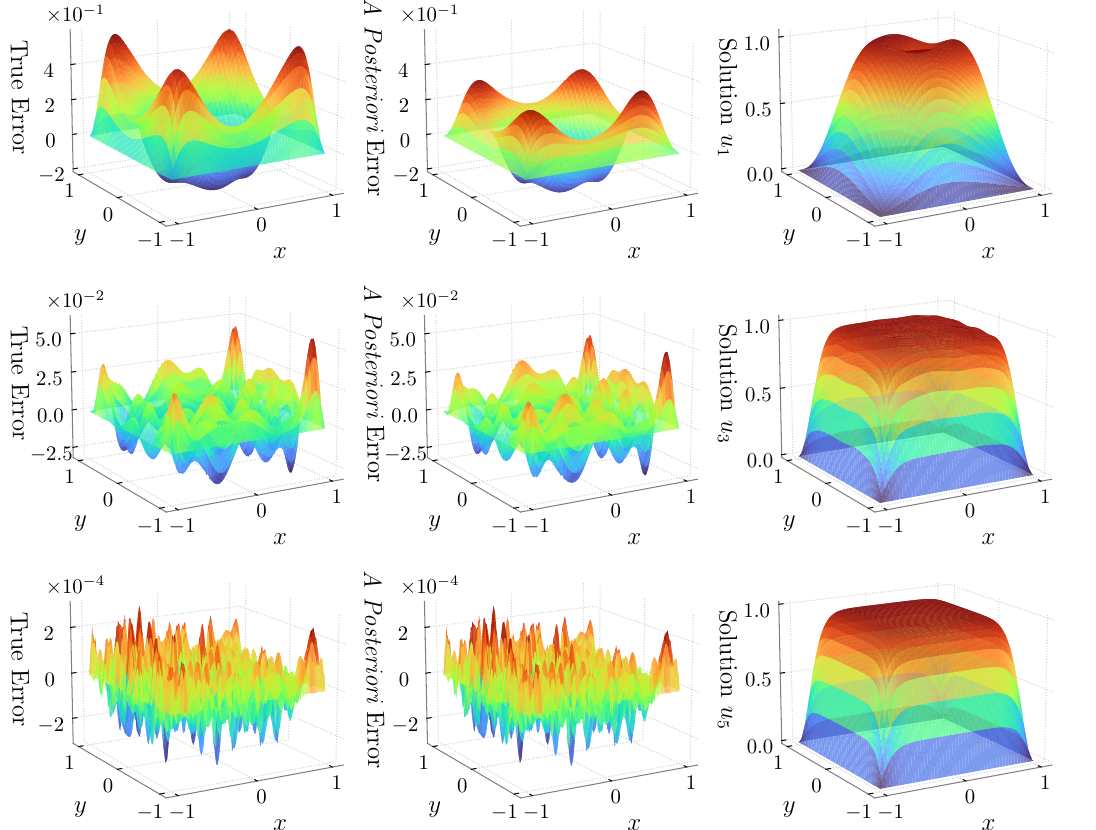}
	\caption{Steady-state Allen-Cahn equation. True errors (left column) and {\it a posteriori} errors (middle column) as well as numerical solutions $u_s$ (right column) for s = 1, 3, 5.}
	\label{fig:2DAllenCahn_new_basis}
\end{figure}

\section{Conclusions}
\label{sec:conclusions}
Deep learning methods have demonstrated great promise in solving PDEs. However, it suffers from accuracy and high computational costs. 
In this paper, we developed an efficient and high accuracy neural network-based approach with robust training for solving PDEs. In particular, we proposed an adaptive method by iteratively constructing a sequence of basis functions using single-hidden-layer neural networks, which can be trained efficiently through the proposed training strategy. Additionally, we developed adaptive strategies for network initialization and collocation sampling, which not only addressed common training challenges but also simplified the tuning of hyperparameters, ensuring stable and accurate performance across a wide range of settings.

We provided an analysis showing that the accuracy of the method improved progressively with each iteration. Furthermore, we derived a reliable {\it a posteriori} error estimator and provided theoretical insights for achieving geometric convergence. We demonstrated both the accuracy and stability of the proposed approach by conducting numerical experiments on various problems, including function fitting, boundary layer problems, the Poisson equation, the biharmonic equation, and the Allen-Cahn equation.
Additionally, the {\it a posteriori} error estimator proved effective in all cases.

The proposed method excels not only in convergence and robustness but also in its broad applicability. Thanks to its collocation-based nature, the algorithm is well-suited to a wide variety of problems, including those involving nonlinearity and complex geometries. Moreover, it can be easily extended to incorporate advanced strategies in tackling specific challenges, further enhancing both accuracy and efficiency, as evidenced by the numerical examples for the Poisson equation. Future work will focus on refining the adaptive basis construction technique and exploring its potential applications to even more complex and higher-dimensional problems.

\appendix
\section{Parameters for Numerical Examples}

This section outlines the details and notations used in the numerical experiments.

\begin{enumerate}[label=$\bullet$, left=5pt]
    \item In all experiments, we set the initial condition as \( u_0 = 0 \) and use \( \tanh \) as the activation function.
    \item Candidate points for rejection sampling are randomly drawn from a fine grid, consisting of 1000 points per dimension. This ensures that the selected sample points are well distributed and not overly clustered.
\end{enumerate}
The following notations are used throughout the experiments:
\begin{enumerate}[label=$\bullet$, left=5pt]
    \item \( M_{\Omega}^1 \), \( M_{\Omega}^2 \), and \( M_{\partial \Omega} \) represent the number of uniform and adaptive collocation points within the domain, as well as the number of boundary points, respectively.
    \item \( \lambda \) is the regularization parameter in the loss function.
    \item \( S \) denotes the total number of stages in the algorithm.
    \item \( N_s \) represents the number of hidden neurons in the SLFN at stage \( s \), where \( s = 1,2,\dots,S \).
    \item \( R_s \) is the scaling factor for the weights at stage \( s \), where \( s = 1,2,\dots,S \).
    \item \( N_{opt} \) is the maximum number of optimization steps in the Adam optimizer.
    \item \( \alpha \) denotes the learning rate in the Adam optimizer.
\end{enumerate}

The detailed parameters used in the numerical experiments are presented in \Cref{tab:paras_common}.
\begin{table}[hbpt]
\centering
\label{tab:paras_common}
\caption{Parameters used in the numerical experiments.}
\begin{tabular}{c||ccccccccc}
\hline\hline
Examples  & $M_{\Omega}^1$ & $M_{\Omega}^2$ & $M_{\partial \Omega}$ & $\lambda$ & $S$   & $N_s$    & $R_s$  & $N_{opt}$ & $\alpha$ \\ \hline\hline
Sec \ref{subsec:function_fitting}                       & 512  & 256    & 2      & 1       & 6   & $5\times2^{s-1}$ & $3s-2$ & 10      & 5e-2   \\ \hline
Sec \ref{subsec:boundary-layer}                         & 512  & 256    & 2      & 10      & 5   & $20\times2^{s}$  & $3s+7$ & 20      & 5e-2   \\ \hline
Sec \ref{subsubsec:L_shaped_domain}                     & 1e4  & 5e3    & 4e3    & 1e3     & 9   & $10\times2^{s}$  & $2s-1$ & 10      & 5e-3   \\ \hline
\makecell{Sec \ref{subsubsec:rapidly_varying_source_term}\\ $\varepsilon = 120^{-1}$} & 1e4  & 5e3    & 4e3    & 1e3     & 11 & $10\times2^{s}$  & $3s-2$ & 10      & 5e-2   \\ \hline
\makecell{Sec \ref{subsubsec:rapidly_varying_source_term}\\ $\varepsilon = 500^{-1}$} & 9e4  & 4.5e4  & 1.6e4  & 1e3     & 13 & $10\times2^{s}$  & $3s-2$ & 10      & 5e-2   \\ \hline
Sec \ref{subsubsec:biharmonic-smooth}                   & 1e4  & 5e3    & 4e3    & 1e3     & 7   & $20\times2^{s}$  & $s+1$  & 10      & 5e-3   \\ \hline
Sec \ref{subsubsec:biharmonic-singular}                 & 1e4  & 5e3    & 4e3    & 1e3     & 8   & $10\times2^{s}$  & $s$    & 10      & 5e-3   \\ \hline
Sec \ref{subsec:Allen_Cahn}                             & 1e4  & 5e3    & 4e2    & 1e3     & 6   & $20\times2^{s}$  & $s$    & 10      & 5e-3   \\ \hline\hline
\end{tabular}
\end{table}
In Section \ref{subsubsec:L_shaped_domain}, which examines the Poisson equation on an L-shaped domain, the knowledge-based test incorporates singular terms of the form \( r^{\lambda_i} \sin( \lambda_i \theta ) \) for \( i \) up to \( N_k = 20 \) in each training stage. All other parameters remain consistent with those used in the standard case, as listed in \Cref{tab:paras_common}.
In Section \ref{subsubsec:rapidly_varying_source_term}, additional parameters for localized training are provided in \Cref{tab:poisson_loc_para}.
\begin{table}[hbpt]
    \centering
    \label{tab:poisson_loc_para}
    \caption{Parameters for adaptive localized neuron strategy used in Section \ref{subsubsec:rapidly_varying_source_term}.}
    \begin{tabular}{c|cccc}
    \hline
         Cases & $S_1$ & $S_2$ & $N_L$ & $R_L$\\
         \hline
         $\varepsilon = 120^{-1}$ & 8 & 3 & 1500 & 10\\
         $\varepsilon = 500^{-1}$ & 8 & 5 & 1500 & 10\\ 
         \hline
    \end{tabular}
\end{table}
In Section \ref{subsec:Allen_Cahn}, which considers the steady-state Allen-Cahn equation, the number of iterations per stage is set as \( I_s = s+1 \).

\section*{Acknowledgments}
This work is supported by the National Natural Science Foundation of China Grants  12171404, W2431008 and 12371409.

\bibliographystyle{siamplain}
\bibliography{references}
\end{document}